\documentclass[11pt,twoside]{amsart}
\usepackage{latexsym,amssymb,amsmath}

\numberwithin{equation}{section}
\newtheorem{theorem}{Theorem}[section]
\newtheorem{lemma}[theorem]{Lemma}
\newtheorem{proposition}[theorem]{Proposition}
\newtheorem{corollary}[theorem]{Corollary}

\theoremstyle{definition}
\newtheorem{definition}[theorem]{Definition} 
 
\newtheorem{remark}[theorem]{Remark}
\newtheorem{example}[theorem]{Example}

\begin{document}


\title[Complete intersection vanishing ideals on degenerate
tori]{Complete intersection vanishing ideals on degenerate tori  
over  finite fields} 

\thanks{The second and third author were partially supported by SNI}

\author{Hiram H. L\'opez}
\address{
Departamento de
Matem\'aticas\\
Centro de Investigaci\'on y de Estudios
Avanzados del
IPN\\
Apartado Postal
14--740 \\
07000 Mexico City, D.F.
}
\email{hlopez@math.cinvestav.mx}

\author{Rafael H. Villarreal}
\address{
Departamento de
Matem\'aticas\\
Centro de Investigaci\'on y de Estudios
Avanzados del
IPN\\
Apartado Postal
14--740 \\
07000 Mexico City, D.F.
}
\email{vila@math.cinvestav.mx}

\author{Leticia Z\'arate}
\address{
CEFyMAP, Universidad Aut\'onoma de Chiapas\\
4a. Oriente Norte No. 1428. Entre 13a. y 14a. Norte\\
Col. Barrio La Pimienta\\ 
29000 Tuxtla Guti\'errez, Chiapas, M\'exico \\ 
}\email{leticia@math.cinvestav.mx}


\begin{abstract} We study the complete
intersection property and the algebraic invariants (index of regularity,
degree) of vanishing ideals on degenerate tori over finite fields. We establish a
correspondence 
between vanishing ideals and toric ideals associated 
to numerical semigroups. This correspondence is shown to preserve the complete
intersection property, and allows us to use some available algorithms
to determine whether  
a given vanishing ideal is a complete intersection. We give formulae for the
degree, and for the index of 
regularity of a complete intersection in terms of the 
Frobenius number and the generators of a numerical semigroup.
\end{abstract}

\maketitle 

\section{Introduction}\label{intro-ci-mcurves}

Let $K=\mathbb{F}_q$  be a finite field with $q$ elements and let
$v_1,\ldots,v_n$ be a sequence of 
positive integers. Consider the {\it degenerate projective torus\/} 
$$
X:=\{[(x_1^{v_{1}},\ldots,x_n^{v_n})]\, \vert\, x_i\in K^*\mbox{ for all
}i\}\subset\mathbb{P}^{n-1},
$$
parameterized by the monomials $x_1^{v_1},\ldots,x_n^{v_n}$, where 
$K^*=\mathbb{F}_q\setminus\{0\}$ and 
$\mathbb{P}^{n-1}$ is a projective space over the field $K$. This set
is a multiplicative group under componentwise multiplication.  
If $v_i=1$ for all $i$, $X$ is just a {\it projective torus\/}. 

Let $S=K[t_1,\ldots,t_n]=\oplus_{d=0}^\infty S_d$ be a polynomial
ring over the field 
$K$ with the standard grading. Recall that the {\it vanishing
ideal\/} of $X$, denoted by 
$I(X)$, is the ideal 
of $S$ generated by the homogeneous polynomials
that vanish on $X$. To study $I(X)$, we will associate with this a
semigroup $\mathcal{S}$ and a toric
ideal $P$ that depend on $v_1,\ldots,v_n$ and the multiplicative
group of $\mathbb{F}_q$. 

In what follows $\beta$ denotes a generator of the cyclic group
$(K^*,\, \cdot\, )$, $d_i$ denotes $o(\beta^{v_i})$, the
order of $\beta^{v_i}$ for $i=1,\ldots,n$, and $\mathcal{S}$ denotes
the semigroup $\mathbb{N}d_1+\cdots+\mathbb{N}d_n$. If
$d_1,\ldots,d_n$ are relatively prime, $\mathcal{S}$ is called a {\it
numerical semigroup\/}. As is seen in
Section~\ref{ci-section}, the algebra of $I(X)$ is closely related to
the algebra of the {\it toric ideal\/} of the  
semigroup ring 
$$K[\mathcal{S}]=K[y_1^{d_1},\ldots,y_1^{d_n}]\subset K[y_1],$$
where $K[y_1]$ is a polynomial ring. Recall that the 
{\it toric ideal\/} of $K[\mathcal{S}]$, denoted by $P$, is the
kernel of the following epimorphism of $K$-algebras
$$\varphi\colon S=K[t_1,\ldots,t_n]\longrightarrow
K[\mathcal{S}],\ \ \ \ \ 
f\stackrel{\varphi}{\longmapsto} f(y_1^{d_1},\ldots,y_1^{d_n}).$$ 

Thus, $S/P\simeq K[\mathcal{S}]$. Since $K[y_1]$ is integral over
$K[\mathcal{S}]$ 
we have ${\rm ht}(P)=n-1$. The ideal $P$ is graded if one gives
degree $d_i$ to variable $t_i$. For $n=3$, the first non-trivial 
case, this type of toric ideals were studied by Herzog
\cite{He3}. For $n\geq 4$, these toric ideals have been studied by many
authors \cite{stcib-algorithm,stcib,delorme,Eli1,ElVi,thoma}.

In this paper, we relate some of the algebraic invariants and
properties of $I(X)$ with those of $P$ and $\mathcal{S}$. We are
especially interested in the  
degree and the regularity index, and in the 
complete intersection property.

The most well known properties that $P$ and $I(X)$ have in common is
that both are Cohen-Macaulay 
graded lattice ideals of dimension $1$ \cite{He3,algcodes}.

The contents of this paper are as follows. In Section~\ref{prelim}, we
introduce some of the notions that will be needed
throughout the paper. 

A key fact that allows us to link the properties of $P$ and $I(X)$
is that the homogeneous lattices of these ideals are closely  
related (Proposition~\ref{jan2-12}). 
If $g_1,\ldots,g_m$ is a set of generators for $P$ consisting of
binomials, then $h_1,\ldots,h_m$ is a set of generators for $I(X)$, where
$h_k$ is the binomial obtained from $g_k$ after substituting $t_i$ by
$t_i^{d_i}$ for $i=1,\ldots,n$ (Proposition~\ref{dec22-11}). As a
consequence if $n=3$, then $I(X)$ is minimally generated by $2$ or $3$
binomials (Corollary~\ref{jun25-12}). If $I(X)$ is a complete
intersection, one of our main results show that a minimal generating set for $I(X)$
consisting of binomials corresponds to a minimal generating set for $P$
consisting of binomials and viceversa (Theorem~\ref{ci-i(x)-p-1}). As
a consequence $I(X)$ is a complete 
intersection if and only if $P$ is a complete intersection
(Corollary~\ref{ci-i(x)-p}).  

We show a formula for the
degree of $S/I(X)$ (Lemma~\ref{dec12-11}).
The {\it Frobenius number\/} of a numerical semigroup is the largest 
integer not in the semigroup.  For complete intersections, we give a formula
that relates the index of regularity
of $S/I(X)$ with the Frobenius number of 
the numerical semigroup generated by
$o(\beta^{rv_1}),\ldots,o(\beta^{rv_n})$, where $r$ is the greatest
common divisor of  
$d_1,\ldots,d_n$ (Corollary~\ref{ci-formula-gcd=1}). 

The Frobenius
number occurs in many branches of mathematics and is one of the most
studied invariants in the theory of 
semigroups. A great deal of effort has been directed at the 
effective computation of this number, see the monograph of Ram\'\i
rez-Alfons\'\i n \cite{frobenius-problem}.  

The complete intersection property of $P$ has been nicely
characterized, using the notion of a binary tree
\cite{stcib-algorithm,stcib} and the notion of 
{\it suites distingu\'ees\/} \cite{delorme}. For $n=3$, there is a
classical result of  
\cite{He3} showing an algorithm to construct a generating set for
$P$. Thus, using our results,
one can obtain various classifications of the complete intersection
property of $I(X)$. Furthermore, in
\cite{stcib-algorithm} an effective algorithm is given to determine
whether $P$ is a complete intersection. This algorithm has been
implemented in the 
distributed library
cimonom.lib \cite{stcib-algorithm-1} of {\it Singular\/} \cite{singular}.
Thus, using our results, one can use this algorithm to
determine whether $I(X)$ is a complete intersection (see
Example~\ref{jun20-12}). If $I(X)$ is a complete intersection, 
this algorithm returns the generators of $P$ and the Frobenius number. As a byproduct,
we can construct interesting examples of complete intersection 
vanishing ideals (see Example~\ref{jan5-12-1}). 

We show how to compute the
vanishing ideal $I(X)$ using the 
notion of saturation of an ideal with respect to a polynomial
(Proposition~\ref{computing-i(x)-saturation}). 

It is worth mentioning that our results could be applied to coding
theory. The algebraic invariants 
of $S/I(X)$ occur in algebraic coding 
theory as we now briefly explain. An {\it evaluation code} over $X$
is a linear code 
obtained by evaluating the linear space of homogeneous $d$-forms of
$S$ on the set of points $X\subset{\mathbb P}^{n-1}$. A linear code
obtained in this way, denoted by 
$C_X(d)$, has {\it length\/} $|X|$ and {\it dimension} 
$\dim_K (S/I(X))_d$. The computation of the index of regularity of
$S/I(X)$ is important 
for applications to 
coding theory: for $d\geq {\rm reg}\, S/I(X)$ the code $C_X(d)$ coincides
with the underlying vector space $K^{|X|}$ and has, 
accordingly, minimum distance equal to $1$. Thus, potentially 
good codes $C_X(d)$ can occur only if $1\leq d < {\rm reg}(S/I(X))$. 
The length, dimension and minimum distance of evaluation codes
$C_X(d)$ arising from complete 
intersections have 
been studied in 
\cite{duursma-renteria-tapia,gold-little-schenck,
hansen,cartesian-codes,affine-codes,ci-codes,d-compl}.

For all unexplained
terminology and additional information,  we refer to \cite{EisStu}
(for the theory of 
lattice ideals), \cite{Sta1,monalg} (for commutative algebra and the
theory of Hilbert functions).

\section{Preliminaries}\label{prelim}

We continue to use the notation and definitions used in
Section~\ref{intro-ci-mcurves}. In this section, we 
introduce the notions of degree and regularity via Hilbert 
functions, and the notion of a lattice ideal. 

The {\it Hilbert function\/} of $S/I(X)$ is given by $H_X(d):=
\dim_K(S_d/I({X})\cap S_d)$, and the {\it Krull-dimension\/} of
$S/I(X)$ is denoted by 
$\dim(S/I(X))$. The unique polynomial 
$$
\textstyle h_X(t)=\sum_{i=0}^{k-1}c_it^i\in 
\mathbb{Q}[t]$$
 of degree $k-1=\dim(S/I(X))-1$ such that
$h_X(d)=H_X(d)$ for 
$d\gg 0$ is called the {\it Hilbert polynomial\/} of $S/I(X)$. The
integer $c_{k-1}(k-1)!$, denoted by ${\rm deg}(S/I(X))$, is 
called the {\it degree\/} of $S/I(X)$. 
According to \cite[Lecture 13]{harris}, $h_X(d)=|X|$ for
$d\geq |X|-1$. Hence
$$
|X|=h_X(d)=c_0=\deg(S/I(X))
$$
for $d\geq |X|-1$. Thus, $|X|$ is the degree of $S/I(X)$.

\begin{definition}\label{definition: regularity}
The \emph{index of regularity} of $S/I(X)$, denoted by
${\rm reg}(S/I(X))$, is the least integer $\ell\geq 0$ such that
$h_X(d)=H_X(d)$ for $d\geq \ell$. 
\end{definition}

The index of regularity of $S/I(X)$ is equal to the Castelnuovo
Mumford regularity of $S/I(X)$ because this ring is Cohen-Macaulay of
dimension $1$. 

\begin{remark}\label{referee2-1} The Hilbert series of
$S/I(X)$ can be 
written as
$$
F_X(t):=\sum_{i=0}^{\infty}H_X(i)t^i=
\frac{h_0+h_1t+\cdots+h_rt^r}{1-t},
$$
where $h_0,\ldots,h_r$ are positive integers. The number $r$ is the
regularity index of $S/I(X)$ and $h_0+\cdots+h_r$ is the degree of
$S/I(X)$ (see \cite[Corollary~4.1.12]{monalg}). The same observation
holds for any graded Cohen-Macaulay ideal $I\subset S$ of height $n-1$.
\end{remark}

Recall that a binomial in $S$ is a polynomial of the
form $t^a-t^b$, where 
$a,b\in \mathbb{N}^n$ and where, if
\mbox{$a=(a_1,\dots,a_n)\in\mathbb{N}^n$}, we set 
\[
t^a=t_1^{a_1}\cdots t_n^{a_n}\in S. 
\]
A {\it binomial ideal\/} is an ideal generated by binomials. 

Given $c=(c_i)\in {\mathbb Z}^n$, the set ${\rm supp}(c)=\{i\, |\,
c_i\neq 0\}$ is the {\it support\/} of
$c$. The vector $c$ can be written as $c=c^+-c^-$, 
where $c^+$ and $c^-$ are two nonnegative vectors 
with disjoint support. If $t^a$ is a monomial,
with $a=(a_i)\in\mathbb{N}^n$, the set 
${\rm supp}(t^a)=\{t_i\vert\, a_i>0\}$ is called the {\it support\/}
of $t^a$.

\begin{definition}\label{lattice-ideal-def}\rm 
A subgroup $\mathcal{L}$ of $\mathbb{Z}^n$ is
called a {\it lattice\/}.   A {\it lattice ideal\/} is an 
ideal of the form
$$
I(\mathcal{L})=(\{t^{\alpha^+}-t^{\alpha^-}\vert\, 
\alpha\in\mathcal{L}\})\subset S
$$
for some lattice $\mathcal{L}$ in $\mathbb{Z}^n$. A  lattice
$\mathcal{L}$ is 
called {\it homogeneous\/} if there is 
an integral vector $\omega$ with positive entries such that
$\langle\omega,a\rangle=0$ for $a\in\mathcal{L}$.
\end{definition}

\begin{definition} An ideal $I\subset S$ is called a {\it complete
intersection\/} if  there exists  $g_1,\ldots,g_m$ such that $I=(g_1,\ldots,g_m)$, 
where $m$ is the height of $I$. 
\end{definition}

\begin{remark}\label{future-use} A graded binomial ideal $I\subset S$
is a complete 
intersection if and only if $I$ is generated by a set of homogeneous
binomials $g_1,\ldots,g_m$, with $m={\rm ht}(I)$, and 
any such set of homogeneous generators is already a regular sequence 
(see \cite[Proposition~1.3.17, Lemma~1.3.18]{monalg}). 
\end{remark}

\begin{lemma}\label{hilbertseries-ci-coro} 
Let $S=K[t_1,\ldots,t_n]$ be a 
polynomial ring with the standard grading. If $I$ is a graded
ideal of $S$ generated by a homogeneous regular sequence
$f_1,\ldots,f_{n-1}$, then 
\[
{\rm reg}(S/I)=\sum_{i=1}^{n-1}(\deg(f_i)-1)\ \mbox{ and }\ {\rm
deg}(S/I)=\deg(f_1)\cdots\deg(f_{n-1}). 
\] 
\end{lemma}

\begin{proof} We set $\delta_i=\deg(f_i)$. By \cite[p.~104]{monalg},
the Hilbert series of $S/I$ 
is given by 
\begin{equation}
F_I(t)=
\frac{\prod_{i=1}^{n-1}\left(1-t^{\delta_i}\right)}{(1-t)^n}=
\frac{\prod_{i=1}^{n-1}(1+t+\cdots+t^{\delta_i-1})}{(1-t)}. 
\end{equation}
Thus, by Remark~\ref{referee2-1}, ${\rm
reg}(S/I)=\sum_{i=1}^{n-1}(\delta_i-1)$ and ${\rm
deg}(S/I)=\delta_1\cdots\delta_{n-1}$. 
\end{proof}

\section{Complete intersections and algebraic invariants}\label{ci-section}

We continue to use the notation and definitions used in
Sections~\ref{intro-ci-mcurves} and \ref{prelim}. In this section, we
study vanishing 
ideals over degenerate projective tori. We study the complete
intersection property and the 
algebraic invariants of vanishing ideals. We will establish a
correspondence between vanishing ideals and toric ideals associated 
to semigroups of $\mathbb{N}$.

Let $D$ be the non-singular matrix $D={\rm diag}(d_1,\ldots,d_n)$.
Consider the homomorphisms of $\mathbb{Z}$-modules:
\begin{eqnarray*} 
\psi\colon\mathbb{Z}^n\rightarrow\mathbb{Z},\ \ \ \ \ \ \ & e_i\mapsto
d_i,\ \ \ \ \ \ \  & \\
D\colon\mathbb{Z}^n\rightarrow\mathbb{Z}^n,& e_i\mapsto d_ie_i.&
\end{eqnarray*}

If $c=(c_i)\in\mathbb{R}^n$, we set $|c|=\sum_{i=1}^nc_i$. Notice 
that $|D(c)|=\psi(c)$ for any $c\in\mathbb{Z}^n$. There are two
homogeneous lattices that will play a role here: 
$$
\mathcal{L}_1=\ker(\psi)\ \mbox{ and } \mathcal{L}=D(\ker(\psi)).
$$
The map $D$ induces a $\mathbb{Z}$-isomorphism between
$\mathcal{L}_1$ and $\mathcal{L}$. It is well known \cite{monalg} that 
the toric ideal $P$ is the lattice ideal of $\mathcal{L}_1$. Below, we show
that $I(X)$ is the lattice ideal of $\mathcal{L}$.

\begin{lemma}\label{dec19-11} The map $t^a-t^b\mapsto
t^{D(a)}-t^{D(b)}$ induces a 
bijection between the binomials $t^a-t^b$ of $P$ whose terms $t^a$,
$t^b$ have disjoint support and the binomials $t^{a'}-t^{b'}$ of
$I(X)$ whose terms 
$t^{a'}$, $t^{b'}$ have disjoint support. 
\end{lemma}

\begin{proof} If $f=t^a-t^b$ is a binomial of $P$ whose terms have
disjoint support, then $a-b\in\mathcal{L}_1$ and the terms of
$g=t^{D(a)}-t^{D(b)}$ have disjoint support because 
$${\rm supp}(t^a)={\rm supp}(t^{D(a)})\  \mbox{ and }\ {\rm supp}(t^b)={\rm
supp}(t^{D(b)}).
$$ 
Thus, $|D(a)|=\psi(a)=\psi(b)=|D(b)|$. This means that $g=t^{D(a)}-t^{D(b)}$
is homogeneous in the standard grading of $S$. As
$(\beta^{v_i})^{d_i}=1$ for all $i$, 
it is seen that
$g$ vanishes at all points of  $X$. Hence, $g\in I(X)$ and the 
map is well defined. 

The map is clearly
injective. To show that the map is onto, take a binomial
$f'=t^{a'}-t^{b'}$ in $I(X)$ with $a'=(a_i')$, $b'=(b_i')$ and such
that $t^{a'}$ and $t^{b'}$ have disjoint support. Then,
$(\beta^{v_i})^{a_i'-b_i'}=1$ for all $i$ because $f'$ vanishes at all
points of $X$. Hence, since the order of $\beta^{v_i}$ 
is $d_i$, there are integers 
$c_1,\ldots,c_n$ such that $a_i'-b_i'=c_id_i$ for all $i$. Since $f'$ is
homogeneous, one has $|a'|=|b'|$. It follows readily that
$c\in\mathcal{L}_1$ and $a'-b'=D(c)$. We can write $c=c^+-c^-$. As 
$a'$ and $b'$ have disjoint support, we get $a'=D(c^+)$ and
$b'=D(c^-)$. Thus, the binomial $f=t^{c^+}-t^{c^{-}}$ is in $P$ and
maps to $t^{a'}-t^{b'}$.
\end{proof}

\begin{proposition}\label{jan2-12}
$P=I(\mathcal{L}_1)$ and $I(X)=I(\mathcal{L})$.
\end{proposition}

\begin{proof} As mentioned above, the first equality is well 
known \cite{monalg}. Since $I(X)$ is a lattice ideal \cite{algcodes}, it is generated
by binomials of the form $t^{a^+}-t^{a^-}$ (this follows using that 
$t_i$ is a non-zero divisor of $S/I(X)$ for all $i$). To show the second equality, take
$t^{a^+}-t^{a^-}$ in $I(X)$. 
Then, by Lemma~\ref{dec19-11}, $a^+-a^-\in\mathcal{L}$ and $t^{a^+}-t^{a^-}$ is
in $I(\mathcal{L})$. Thus, $I(X)\subset I(\mathcal{L})$. 
Conversely, take $f=t^{a^+}-t^{a^-}$ in
$I(\mathcal{L})$ with $a^+-a^-$ in $\mathcal{L}$. Then, there is
$c\in\mathcal{L}_1$ such that $a^+-a^-=D(c^+-c^-)$. Then,
$t^{c^+}-t^{c^-}$ is in $P$ and maps, under the map of
Lemma~\ref{dec19-11}, to $f$. Thus, $f\in I(X)$. This proves that
$I(\mathcal{L})\subset I(X)$.
\end{proof}

\begin{proposition}\label{dec22-11}
If $P=(\{t^{a_i}-t^{b_i}\}_{i=1}^m)$, then
$I(X)=(\{t^{D(a_i)}-t^{D(b_i)}\}_{i=1}^m)$.
\end{proposition}
\begin{proof} We set $g_i=t^{a_i}-t^{b_i}$ and
$h_i=t^{D(a_i)}-t^{D(b_i)}$ for $i=1,\ldots,n$. Notice that
$h_i$ is equal to $g_i(t^{d_1},\ldots,t^{d_n})$, the evaluation of
$g_i$ at $(t_1^{d_1},\ldots,t_n^{d_n})$. By Lemma~\ref{dec19-11}, one has
the inclusion $(h_1,\ldots,h_m)\subset I(X)$. To show the reverse
inclusion take a binomial $0\neq f\in I(X)$. We may assume that
$f=t^{a^+}-t^{a^-}$. Then, by Lemma~\ref{dec19-11}, there is
$g=t^{c^+}-t^{c^-}$ in $P$ such that $f=t^{D(c^+)}-t^{D(c^-)}$. By
hypothesis we can write $g=\sum_{i=1}^mf_ig_i$ for some
$f_1,\ldots,f_m$ in $S$. Then, evaluating both sides of this equality at
$(t_1^{d_1},\ldots,t_n^{d_n})$, we get
$$
f=t^{D(c^+)}-t^{D(c^-)}=g(t_1^{d_1},\ldots,t_n^{d_n})
=\sum_{i=1}^mf_i(t_1^{d_1},\ldots,t_n^{d_n})g_i(t_1^{d_1},\ldots,t_n^{d_n})=
\sum_{i=1}^mf_i'h_i,
$$
where $f_i'=f_i(t_1^{d_1},\ldots,t_n^{d_n})$ for all $i$. Then, $f\in
(h_1,\ldots,h_m)$. 
\end{proof}

\begin{corollary}\label{jun25-12} If $n=3$, then $I(X)$ is minimally generated 
by at most $3$ binomials. 
\end{corollary}

\begin{proof} By a classical theorem of Herzog \cite{He3}, $P$ is
generated by at most $3$ binomials. Hence, by
Proposition~\ref{dec22-11}, $I(X)$ is generated by at most $3$
binomials. 
\end{proof}

Given a subset $I\subset S$, its {\it variety\/}, denoted by $V(I)$,
is the set of all  
$a\in\mathbb{A}_K^n$ such that $f(a)=0$ for all $f\in I$, where
$\mathbb{A}_K^n$ is the affine space over $K$. Given a
binomial $g=t^a-t^b$, we set $\widehat{g}=a-b$. If $B$ is a 
subset of $\mathbb{Z}^n$, $\langle B\rangle$ denotes the subgroup of 
$\mathbb{Z}^n$ generated by $B$.

\begin{proposition}{\rm\cite[Proposition~2.5]{stcib}}\label{LaConcepcion-dec24-2011}
Let $\mathcal{B}=\{g_1,\ldots,g_{n-1}\}$ be a set of binomials in $P$.
Then, $P=(\mathcal{B})$ if and only if the following two conditions hold\/{\rm :}
\begin{itemize}
\item[$(\mathrm{i}')$] $\mathcal{L}_1=\langle
\widehat{g}_1,\ldots,\widehat{g}_{n-1}\rangle$, where
$\mathcal{L}_1=\ker(\psi)$. 
\item[$(\mathrm{ii}')$] $V(g_1,\ldots,g_{n-1},t_i)=\{0\}$ for $i=1,\ldots,n$.
\end{itemize}
\end{proposition}

We come to the main result of this section.

\begin{theorem}\label{ci-i(x)-p-1} 
$(\mathrm{a})$ If $I(X)$ is a complete intersection generated by 
binomials
$h_1,\ldots,h_{n-1}$, then $P$ is a complete intersection 
generated by binomials $g_1,\ldots,g_{n-1}$ such that
$h_i$ is equal to $g_i(t_1^{d_1},\ldots,t_n^{d_n})$ for all $i$.
$(\mathrm{b})$ If $P$ is a complete intersection generated by
binomials  
$g_1,\ldots,g_{n-1}$, then $I(X)$ is a complete intersection 
generated by binomials $h_1,\ldots,h_{n-1}$, where 
$h_i$ is equal to $g_i(t_1^{d_1},\ldots,t_n^{d_n})$ for all $i$.  
\end{theorem}

\begin{proof} (a) Since $t_k$ is a non-zero divisor of $S/I(X)$ for
all $k$, it is not hard to see that the monomials of $h_i$ have
disjoint support for all $i$, i.e., we can write $h_i=t^{a_i^+}-t^{a_i^-}$ for
$i=1,\ldots,n-1$. We claim that the following two conditions hold.
\begin{itemize}
\item[$(\mathrm{i})$]
$\mathcal{L}=\langle a_1,\ldots,a_{n-1}\rangle$, where $a_i=a_i^+-a_i^-$
and $\mathcal{L}$ is the lattice that defines $I(X)$. 
\item[$(\mathrm{ii})$] $V(h_1,\ldots,h_{n-1},t_i)=\{0\}$ for $i=1,\ldots,n$.
\end{itemize} 

As $I(X)$ is generated by $h_1,\ldots,h_{n-1}$, by
\cite[Lemma~2.5]{ci-lattice}, condition (i) holds. The binomial
$t_i^{q-1}-t_n^{q-1}$ is in $I(X)$ for all $i$ because
$\mathbb{F}_q^*$ is a group of order $q-1$. Thus, 
$V(I(X),t_i)=\{0\}$ for all $i$. From the equality
$(h_1,\ldots,h_{n-1},t_i)=(I(X),t_i)$,
we get
$$
V(h_1,\ldots,h_{n-1},t_i)=V(I(X),t_i)=\{0\}. 
$$
Thus, (ii) holds. This completes the proof of the claim. 

By (i) and Proposition~\ref{jan2-12}, there are $b_1,\ldots,b_{n-1}$
in $\mathcal{L}_1={\rm ker}(\psi)$ such that
$a_i=D(b_i)$ for all $i$. Accordingly $a_i^+=D(b_i^+)$ and
$a_i^-=D(b_i^-)$ for all $i$. We set $g_i=t^{b_i^+}-t^{b_i^-}$ for all
$i$. Clearly, all the $g_i$'s are in $P$ and $h_i$ is equal to
$g_i(t_1^{d_1},\ldots,t_n^{d_n})$ for all $i$. Next, we prove that $P$ is generated by
$g_1,\ldots,g_{n-1}$. By Proposition~\ref{LaConcepcion-dec24-2011} it
suffices to show that the following two conditions hold:
\begin{itemize}
\item[$(\mathrm{i}')$] $\mathcal{L}_1=\langle
b_1,\ldots,b_{n-1}\rangle$, where
$\mathcal{L}_1=\ker(\psi)$. 
\item[$(\mathrm{ii}')$] $V(g_1,\ldots,g_{n-1},t_i)=\{0\}$ for $i=1,\ldots,n$.
\end{itemize}

First we show $(\mathrm{i}')$. Since $b_1,\ldots,b_{n-1}$ are in $\mathcal{L}_1$,
we need only show the inclusion ``$\subset$''. Take $\gamma\in{\rm
ker}(\psi)$, then $D(\gamma)\in 
\mathcal{L}$, and by (i) it follows that $\gamma\in \langle
b_1,\ldots,b_{n-1}\rangle$. 

Next we show $(\mathrm{ii}')$. For simplicity of notation,
we may assume that $i=n$. Take $c$ in the variety
$V(g_1,\ldots,g_{n-1},t_n)$ and write $c=(c_1,\ldots,c_n)$.  Then, $c_n=0$ and
$g_i(c)=c^{b_i^+}-c^{b_i^-}=0$ for all $i$,
were $c^{b_i^+}$ means to evaluate the monomial $t^{b_i^+}$ 
at the point $c$. Let $i$ be a fixed but arbitrary integer in 
$\{1,\ldots,n-1\}$. We can write
$$
b_i=b_i^+-b_i^-=(b_{i1}^+,\ldots,b_{in}^+)-(b_{i1}^-,\ldots,b_{in}^-)
$$
and
$a_i=a_i^+-a_i^-=(a_{i1}^+,\ldots,a_{in}^+)-(a_{i1}^-,\ldots,a_{in}^-)$.  
Then
\begin{eqnarray}
h_i(c_1^{v_1},\ldots,c_n^{v_n})&=&(c_1^{v_1})^{a_{i1}^+}\cdots
(c_n^{v_n})^{a_{in}^+}-(c_1^{v_1})^{a_{i1}^-}\cdots
(c_n^{v_n})^{a_{in}^-}\nonumber\\
&=&c_1^{v_1d_1b_{i1}^+}\cdots
c_n^{v_nd_nb_{in}^+}-c_1^{v_1d_1b_{i1}^-}\cdots
c_n^{v_nd_nb_{in}^-}.\label{jun17-12}
\end{eqnarray}
We claim that $h_i(c_1^{v_1},\ldots,c_n^{v_n})=0$. To show this we
consider two cases. 

Case (I): $b_{in}^+>0$. Then, as $g_i(c)=c^{b_i^+}-c^{b_i^-}=0$ and
$c^{b_i^+}$=0, 
one has $c^{b_i^-}=0$. Hence, there is $j$ such that $b_{ij}^->0$ and
$c_j=0$. Thus, by Eq.~(\ref{jun17-12}), $h_i(c_1^{v_1},\ldots,c_n^{v_n})=0$.

Case (II): $b_{in}^+=0$. If $c_j=0$ for some $b_{ij}^+>0$, then 
$c^{b_i^-}=0$ because $g_i(c)=0$. Hence, there is $k$ such that
$c_k=0$ and $b_{ik}^->0$. Thus, by Eq.~(\ref{jun17-12}),
$h_i(c_1^{v_1},\ldots,c_n^{v_n})=0$. 
Similarly, if $c_j=0$ for some $b_{ij}^->0$, then  
$c^{b_i^+}=0$ because $g_i(c)=0$. Hence, there is $k$ such that
$c_k=0$ and $b_{ik}^+>0$. Thus, by Eq.~(\ref{jun17-12}),
$h_i(c_1^{v_1},\ldots,c_n^{v_n})=0$.  
We may now assume that $c_j\neq 0$ if $b_{ij}^+>0$, and $c_m\neq 0$
if $b_{im}^->0$. Let $\beta$ be a generator of the cyclic group
$(\mathbb{F}^*,\, \cdot\, )$. Any $c_j\neq 0$ has the form $c_j=\beta^{j_\ell}$.
Thus,  using that $(\beta^{v_j})^{d_j}=1$, we get that
$(c_j^{v_j})^{d_jb_{ij}^+}=1$ if $b_{ij}^+>0$ and 
$(c_j^{v_j})^{d_jb_{ij}^-}=1$ if $b_{ij}^->0$. Hence, by
Eq.~(\ref{jun17-12}), $h_i(c_1^{v_1},\ldots,c_n^{v_n})=0$, 
as required. This completes the proof of the claim. 

As $h_i(c_1^{v_1},\ldots,c_n^{v_n})=0$ for all $i$, the point 
$c'=(c_1^{v_1},\ldots,c_n^{v_n})$ is in 
$V(h_1,\ldots,h_{n-1},t_n)$. By (ii), the point $c'$ es zero. Hence, 
$c=0$ as required. This completes the proof of $(\mathrm{ii}')$.
Hence, $P$ is a complete
intersection generated by $g_1,\ldots,g_{n-1}$. 

(b) It follows from Proposition~\ref{dec22-11}. 
\end{proof}

Using the notion of a binary 
tree, a criterion for complete intersection
toric ideals of affine monomial curves is given 
in \cite{stcib}. In
\cite{stcib-algorithm} an effective algorithm is 
given to determine whether $P$ is a complete intersection. If $P$ is 
a complete intersection, this algorithm returns the generators of $P$
and the Frobenius number. 

In our situation, the next result allows us to: (A) use the
results of \cite{stcib,delorme,He3} to give criteria for complete intersection
vanishing ideals over a finite field, (B) use the effective 
algorithms of \cite{stcib-algorithm} to recognize complete
intersection vanishing ideals over finite fields and to compute its 
invariants (see Example~\ref{jun20-12}).

\begin{corollary}\label{ci-i(x)-p}  $I(X)$ is a complete intersection if
and only if $P$ is a complete intersection.
\end{corollary}

\begin{proof} Assume that $I(X)$ is a complete intersection. By
Remark~\ref{future-use}, there are binomials $h_1,\ldots,h_{n-1}$ that
generate $I(X)$. Hence, $P$ is a complete intersection by
Theorem~\ref{ci-i(x)-p-1}. The converse follows by similar reasons. 
\end{proof}

\begin{lemma}\label{jan1-12} If $r=\gcd(d_1,\ldots,d_n)$ and
$d_i'=o(\beta^{rv_i})$, then $d_i=rd_i'$ and 
$\gcd(d_1',\ldots,d_n')=1$. 
\end{lemma}

\begin{proof} It follows readily by recalling that 
$o(\beta^{rv_i})=o(\beta^{v_i})/\gcd(r,o(\beta^{v_i}))$.
\end{proof}

In what follows $X'$ will denote the degenerate torus in 
$\mathbb{P}^{n-1}$ parameterized by $x_1^{v_1'},\ldots,x_n^{v_n'}$, where 
$v_i'=rv_i$ and $r=\gcd(d_1,\ldots,d_n)$. Below, we relate $I(X)$ and $I(X')$. 

\begin{proposition}\label{ci-x-x'} The vanishing ideal $I(X)$ is a 
complete intersection if and only if  $I(X')$ is a
complete intersection. 
\end{proposition}

\begin{proof} Let $P$ and $P'$ be the toric ideals 
of $K[y_1^{d_1},\ldots,y_1^{d_n}]$ and
$K[y_1^{d_1'},\ldots,y_1^{d_n'}]$, 
respectively, where $d_i'=o(\beta^{rv_i})$ for all $i$. 
It is not hard to see that $P=P'$. Then, by
Theorem~\ref{ci-i(x)-p-1}, $P$ is a
complete intersection if and only if $I(X)$ is a complete intersection
and $P'$ is a complete intersection 
if and only if $I(X')$ is a complete intersection. Thus, $I(X)$ is a 
complete intersection if and only if  $I(X')$ is a
complete intersection.
\end{proof}

\begin{definition}\label{projectivetorus-def} The set 
$X^*:=\{(x_1^{v_{1}},\ldots,x_n^{v_n})\vert\, x_i\in
K^*\mbox{ for all }i\}\subset{K}^{n}$ 
is called an {\it affine degenerate torus\/} 
parameterized by $x_1^{v_1},\ldots,x_n^{v_n}$.
\end{definition}

\begin{lemma}\label{dec12-11} $|X^*|=d_1\cdots d_n$ and
$\deg(S/I(X))=|X|=d_1\cdots d_n/\gcd(d_1,\ldots,d_n)$. 
\end{lemma}

\begin{proof} Let $S_{i}=\langle\beta^{v_i}\rangle$
be the cyclic group generated by $\beta^{v_i}$. The set 
$X^*$ is equal to the cartesian product $S_1\times\cdots\times S_n$. 
Hence, to show the first equality, it suffices to recall that $|S_i|$
is $o(\beta^{v_i})$,  
the order of $\beta^{v_i}$. Notice that any element of $X^*$ can be written as
$((\beta^{i_1})^{v_1},\ldots,(\beta^{i_n})^{v_n})$ for some integers
$i_1,\ldots,i_n$. The kernel of the 
epimorphism  of groups $X^*\mapsto X$, $x\mapsto [x]$, is
equal to 
$$
\{(\gamma,\ldots,\gamma)\in (K^*)^n\colon \gamma\in\langle 
\beta^{v_1}\rangle\cap\cdots\cap\langle\beta^{v_n}\rangle \}.
$$
Hence, $|X^*|/|\cap_{i=1}^n\langle\beta^{v_i}\rangle|=|X|$. Since
$\langle\beta^{v_i}\rangle$ is a subgroup of $K^*$ for all $i$ and
$K^*$ is a cyclic group, one has
$|\cap_{i=1}^n\langle\beta^{v_i}\rangle|=\gcd(d_1,\ldots,d_n)$ (see
for instance \cite[Theorem~4, p.~4]{alperin}). Thus, the second
equality follows.
\end{proof}

\begin{definition}\label{frobenius-number-def} If $\mathcal{S}$ is a
numerical semigroup of 
$\mathbb{N}$, the {\it Frobenius number\/} 
of $\mathcal{S}$, denoted by $g(\mathcal{S})$, is the largest 
integer not in $\mathcal{S}$. 
\end{definition}

Consider the semigroup
$\mathcal{S}'=\mathbb{N}d_1'+\cdots+\mathbb{N}d_n'$, where 
$d_i'=o(\beta^{rv_i})$ for $i=1,\ldots,n$. 
By Lemma~\ref{jan1-12}, one has $\gcd(d_1',\ldots,d_n')=1$, i.e., 
$\mathcal{S}'$ is a numerical semigroup. Thus, $g(\mathcal{S}')$ is
finite. If the toric ideal 
of $K[\mathcal{S}']$ is a complete intersection, 
then $g(\mathcal{S}')$ can be expressed entirely in terms of
$d_1',\ldots,d_n'$ \cite[Remark~4.5]{stcib}.

\begin{corollary}\label{ci-formula-gcd=1} $\mathrm{(i)}$ ${\rm
deg}(S/I(X))=d_1\cdots d_n/\gcd(d_1,\ldots,d_n)$.

$\mathrm{(ii)}$ If $I(X)$ is a complete intersection, then 
$$\textstyle{\rm reg}\,
S/I(X)=\gcd(d_1,\ldots,d_n)\, g(\mathcal{S}')+\sum_{i=1}^n d_i-(n-1).$$
\end{corollary}

\begin{proof} Part (i) follows at once from Lemma~\ref{dec12-11}. Next, 
we prove (ii). Let $P$ and $P'$ be as in the proof of
Proposition~\ref{ci-x-x'}. With the notation above, by Lemma~\ref{jan1-12}, we get
that $d_i=rd_i'$ for all $i$. The toric ideals $P$ and $P'$ are equal but they are
graded differently. Recall that $P$ and $P'$ are graded with respect to
the gradings 
induced by assigning $\deg(t_i)=d_i$ and $\deg(t_i)=d_i'$ 
for all $i$, respectively. Let $g_1,\ldots,g_{n-1}$ be a generating set of $P=P'$
consisting of binomials. Then, by Theorem~\ref{ci-i(x)-p-1}, $I(X)$ is
generated by $h_1,\ldots,h_{n-1}$, where $h_i$ is
$g_i(t_1^{d_1},\ldots,t_n^{d_n})$ for all $i$. Accordingly, $I(X')$ is
generated by $h_1',\ldots,h_{n-1}'$, where $h_i'$ is
$g_i(t_1^{d_1'},\ldots,t_n^{d_n'})$ for all $i$. If $D_i=\deg(h_i)$
and $D_i'=\deg(h_i')$, then $D_i=rD_i'$ for all $i$. As $P'$ is a
complete intersection generated by $g_1,\ldots,g_{n-1}$ and
$\deg_{P'}(g_i)=D_i'$ for all $i$, using 
\cite[Remark~4.5]{stcib}, we get
$$
g(\mathcal{S}')=\sum_{i=1}^{n-1}D_i'-\sum_{i=1}^{n}d_i'=
\sum_{i=1}^{n-1}(D_i/r)-\sum_{i=1}^{n}(d_i/r).
$$
Therefore, using the equality ${\rm reg}\, S/I(X)=\sum_{i=1}^{n-1}(D_i-1)$ (see
Lemma~\ref{hilbertseries-ci-coro}), the
formula for the regularity follows.
\end{proof}

\begin{example}\label{jun20-12} To illustrate how to use the
algorithm of \cite{stcib-algorithm} we consider
the degenerate torus $X$, over the field $\mathbb{F}_q$, parameterized by
$x_1^{v_1},\ldots,x_5^{v_5}$, 
where $v_1=1500$, $v_2=1000$, $v_3=432$, $v_4=360$, $v_5=240$, 
and $q=54001$. In this case, one has 
$$ 
d_1=36,\ d_2=54,\ d_3=125,\ d_4=150,\ d_5=225.
$$

Using \cite[Algorithm CI, p.~981]{stcib-algorithm}, we get that $P$ is
a complete intersection generated by the binomials
$$
g_1=t_1^3-t_2^2,\ g_2=t_4^3-t_5^2,\ g_3=t_3^3-t_4t_5,\
g_4=t_1^8t_2^3-t_4^3,
$$
and we also get that the Frobenius number of $\mathcal{S}$ is $793$.
Hence, by our results, the vanishing ideal $I(X)$ is a complete
intersection generated by the binomials
$$
h_1=t_1^{108}-t_2^{108},\ h_2=t_4^{450}-t_5^{450},\
h_3=t_3^{375}-t_4^{150}t_5^{225},\
h_4=t_1^{288}t_2^{162}-t_4^{450},
$$ 
the index of regularity and degree of $S/I(X)$ are $1379$ and
$8201250000$, respectively.
\end{example}

The next example is interesting because if one chooses $v_1,\ldots,v_n$
at random, it is likely that $I(X)$ will be generated by 
binomials of the form $t_i^m-t_j^m$.

\begin{example}\label{jan5-12} Let $\mathbb{F}_q$ be the field with
$q=211$ elements.  
Consider the sequence $v_1=42$, $v_2=35$, $v_3=30$. In this case, one
has $d_1=5$, $d_2=6$, $d_3=7$. By a well known result of Herzog \cite{He3}, 
one has 
$$
P=(t_2^{2}-t_1t_3,\, t_1^{4}-t_2t_3^{2},\,
t_1^{3}t_2-t_3^{3}).
$$

Hence, by our results, 
$I(X)=(t_2^{12}-t_1^5t_3^7,\, t_1^{20}-t_2^6t_3^{14},\,
t_1^{15}t_2^6-t_3^{21})$ and this ideal is not a complete
intersection. The index of regularity and the degree of $S/I(X)$ are
$25$ and $210$, respectively. The Frobenius number of $\mathcal{S}$ 
is equal to $9$. Notice that the toric
relations  $t_1^{30}-t_2^{30}$, $t_1^{35}-t_3^{35}$,
$t_2^{42}-t_3^{42}$ do not generate $I(X)$. 
\end{example}

The next example was found using Theorem~\ref{ci-i(x)-p-1}. Without 
using this theorem it is very difficult to construct
examples of complete intersection vanishing ideals not generated by
binomials of the form $t_i^{m}-t_j^m$. 

\begin{example}\label{jan5-12-1} Let $\mathbb{F}_q$ be the field with
$q=271$ elements. 
Consider the sequence $v_1=30$, $v_2=135$, $v_3=54$. In this case, one
has $d_1=9$, $d_2=2$, $d_3=5$. The ideals $P$ and $I(X)$ are complete
intersections given by
$$
P=(t_1-t_2^2t_3,\, t_2^{5}-t_3^{2})\ \mbox{ and }\ 
I(X)=(t_1^{9}-t_2^4t_3^5,\, t_2^{10}-t_3^{10}).
$$
By Lemma~\ref{hilbertseries-ci-coro}, the index of regularity of $S/I(X)$
is $17$ and by Corollary~\ref{ci-formula-gcd=1} the Frobenius number
of $\mathcal{S}$ is $3$.  
\end{example}

\subsection*{The computation of the vanishing ideal}

In this part we show how to compute the vanishing ideal using the
notion of saturation of an ideal with respect to a polynomial.

The next lemma is easy to show.

\begin{lemma}\label{jun19-12} If $c_{ij}:={\rm
lcm}\{d_i,d_j\}={\rm lcm}\{o(\beta^{v_i}),o(\beta^{v_j})\}$,  
then $t_i^{c_{ij}}-t_j^{c_{ij}}\in I(X)$.
\end{lemma}

The set of {\it toric relations\/}
$\mathcal{T}=\{t_i^{c_{ij}}-t_j^{c_{ij}}\colon 1\leq i,j\leq n\}$ do not generate
$I(X)$, as is seen in Example~\ref{jan5-12}. If $v_i=1$ for all $i$,
then $c_{ij}=q-1$ for all $i,j$ and $I(X)$ is generated by $\mathcal{T}$.  

For  an ideal $I\subset S$ and  a polynomial $h\in S$ the {\it
saturation\/} of $I$ with respect to $h$
is the ideal
$$(I\colon
h^{\infty}):=\{f\in S\, \vert\, fh^k\in I\mbox{ for some }k\geq 1\}.
$$

\begin{proposition}\label{computing-i(x)-saturation} 
Let $I'$ be the ideal $(t_i^{c_{ij}}-t_j^{c_{ij}}\vert\,
1<i<j\leq n)$, where  $c_{ij}={\rm lcm}\{d_i,d_j\}$. 
If $\gcd(d_1,\ldots,d_n)=1$, then $I(X)=(I'\colon
(t_1\cdots t_n)^\infty)$. 
\end{proposition}

\begin{proof} We claim that $\mathcal{L}=\langle c_{ij}e_i-c_{ij}e_j\vert\,
1\leq i<j\leq n\rangle$. By \cite[Proposition~10.1.8]{monalg}, we get  
$$\mathcal{L}_1=\langle (d_j/\gcd(d_i,d_j))e_i-(d_i/\gcd(d_i,d_j))e_j\vert\,
1\leq i<j\leq n\rangle.$$ 
Thus, the claim follows from the equality
$\mathcal{L}=D(\mathcal{L}_1)$. The inclusion ``$\supset$'' follows 
readily using that $t_i$ is a non-zero divisor of $S/I(X)$ for all $i$
because $I(X)$ is a lattice
ideal containing $I'$ (see Lemma~\ref{jun19-12}). To show
the inclusion ``$\subset$'', take a binomial $f=t^a-t^b\in I(X)$. By 
Proposition~\ref{jan2-12}, $I(X)=I(\mathcal{L})$. Thus, 
$a-b\in\mathcal{L}$. Using the previous claim and \cite[Lemma~2.3]{ci-lattice}, there is
$\delta\in\mathbb{N}^n$ such that $t^\delta f\in I'$. Hence, $f\in (I'\colon
(t_1\cdots t_n)^\infty)$.
\end{proof}

\bibliographystyle{plain}

\end{document}